\documentclass[12pt]{amsart}
\usepackage[utf8]{inputenc} 
\usepackage[active]{srcltx}
\usepackage{a4wide}
\usepackage{amsthm,amsfonts,amsmath,mathrsfs,amssymb}
\usepackage{dsfont}
\usepackage{mathtools}
\usepackage[T1]{fontenc}
\usepackage[utf8]{inputenc}
\usepackage{enumerate}
\usepackage[left=4cm,top=4cm,right=4cm,bottom=4cm]{geometry}
\usepackage{caption}
\usepackage{subcaption}
\numberwithin{equation}{section}

\usepackage{geometry}
\usepackage{graphicx}
\usepackage{amssymb}
\usepackage{amsmath}
\usepackage{amsthm}
\usepackage{listings}
\usepackage[colorlinks=true,urlcolor=blue,citecolor=blue,linkcolor=blue]{hyperref}
\usepackage{esint}
\usepackage{xcolor}
\usepackage{ulem}
\usepackage{cancel}

\usepackage{cite}
\usepackage{etex}
\usepackage[all]{xy}
\usepackage{pgf,tikz}
\usetikzlibrary{arrows}
\usepackage{pgfplots,pgfcalendar}
\usetikzlibrary{patterns}
\usepackage{hyperref}

\def\D{{\mathbb D}}  
\def\C{{\mathbb C}}  \def\N{{\mathbb N}}

\def\({\left(}       \def\){\right)}

\newtheorem{theorem}{Theorem}[section]

\newtheorem{proposition}[theorem]{Proposition}

\theoremstyle{definition}
\newtheorem{definition}[theorem]{Definition}

\theoremstyle{remark}

\numberwithin{equation}{section}

\begin{document}
\title[Mixed norm spaces and $RM(p,q)$ spaces]{Mixed norm spaces and $RM(p,q)$ spaces}
\author[T. Aguilar-Hern\'andez ]{Tanaus\'u Aguilar-Hern\'andez}
\address{Departamento de Matem\'atica Aplicada II and IMUS, Escuela T\'ecnica Superior de Ingenier\'ia, Universidad de Sevilla,
Camino de los Descubrimientos, s/n 41092, Sevilla, Spain}
\email{taguilar@us.es}


\date{\today}


\thanks{{This research was supported in part by Ministerio de Econom\'{\i}a y Competitividad, Spain, MTM2015-63699-P,  and Junta de Andaluc{\'i}a, FQM133.}}

\maketitle

\begin{abstract}
In this paper we present the containment relationship between the spaces of analytic functions with average radial integrability  $RM(p,q)$ and a family of mixed norm spaces.
\end{abstract}


\section{Introduction}

The belonging of a function to a certain Banach space of analytic functions is usually given in terms of boundedness (or integrability) of a certain average of the function on circles centered at the origin or in terms of the integrability with respect to the Lebesgue area measure, maybe with a certain weight (see, i.e., \cite{duren_theory_2000,garnett_bounded_2007,HKZ,vukotic_multiplier_2016}).   The most classical examples in this situation are the Hardy spaces but also the mixed norm spaces $H^{q,p}$ defined explicitly by Flett in \cite{Flett_1,Flett_2} and, nowadays, widely studied  (see \cite{vukotic_multiplier_2016}).  Let us recall that a holomorphic function $f$ in the unit disc belongs to $H^{q,p}$ if 
\begin{equation}\label{Eq:mixta}
\left(\int_{0}^{1} \left( \int_{0}^{2\pi}|f(re^{i\theta})|^{q}\ \frac{d\theta}{2\pi}\right)^{p/q}\ dr\right)^{1/p}<+\infty.
\end{equation}

In other less studied cases, the belonging to a Banach space of analytic functions is determined by the average radial integrability. Maybe the most well-known space in this situation is the spaces of bounded radial variation $BRV$, a topic that goes back to Zygmund and where many different authors have work (see, i.e., \cite{Bourgain,rudin_1955,zygmund_1944}). This space of analytic function with bounded radial variation consists of those holomorphic functions $g\in\mathcal{H}(\D)$ such that 
\begin{align*}
\sup_{\theta}\int_{0}^{1} |g'(te^{i\theta})|\ dt<+\infty.
\end{align*}

Other different situation where the radial integrability plays an important role is in the Féjer-Riesz theorem which says that if $f$ belongs to the Hardy space $H^p$ then
\begin{align*}
\sup_{\theta}\left(\int_{0}^{1} |f(re^{i\theta})|^{p}\ dr\right)\leq \frac{1}{2}\|f\|_{H^p}^{p}.
\end{align*}

Recently, a family of spaces of holomorphic functions in the unit disc with average radial integrability, denoted by $RM(p,q)$, has been studied in \cite{Aguilar-Contreras-Piazza_1, Aguilar-Contreras-Piazza_2}. These spaces are formed by the analytic functions such that 
\begin{equation}\label{Eq:radial}
\left(\int_{0}^{2\pi}\left(\int_{0}^{1} |f(re^{i\theta})|^{p}\ dr\right)^{q/p}\ \frac{d\theta}{2\pi}\right)^{1/q}<+\infty
\end{equation}
for $0<p,q<+\infty$. If either $p$ or $q$ is infinity, we change the integral by the essential supremum, respectively. 
 This family contains the classical Hardy spaces $H^q$ (when $p=+\infty$) and Bergman spaces $A^p$ (when $p=q$). 
 
 Looking at \eqref{Eq:mixta} and \eqref{Eq:radial}, a natural question that has been raised is the containment relationship with the mixed norm spaces $H^{q,p}$ and $RM(p,q)$. Notice that by Fubini's theorem, it is clear that
$RM(p,p)= H^{p,p}$ for all $1<p<\infty$. 
 The main result of this paper will provide an answer to the above question:

\begin{theorem} \label{main} Let $1< p,q< +\infty$.
	\begin{enumerate}
		\item[a)] If $p>q$, then $RM(p,q)\varsubsetneq H^{q,p}$.
		\item[b)] If $q>p$, then $H^{q,p} \varsubsetneq RM(p,q)$.
	\end{enumerate}
\end{theorem}

Throughout the paper the letter $C=C(\cdot)$ will denote an absolute constant whose value depends on the parameters indicated in the parenthesis, and may change from one occurrence to another. We will use the notation $a\lesssim b$ if there exists a constant $C=C(\cdot)>0$ such that $a\leq C b$, and $a\gtrsim b$ is understood in an analogous manner. In particular, if $a\lesssim b$ and $a\gtrsim b$, then we will write $a\asymp b$.

\section{First definitions}
We start by recalling the definition of the spaces with average radial integrability $RM(p,q)$. These spaces are formed by those holomorphic functions in $\D$ such that taking the $p$-norm in every radius and then the $q$-norm, the result is a finite number.
More precisely, we give the following definition.

\begin{definition}
	Let $1\leq <p,q< +\infty$. We define the spaces of analytic functions $$RM(p,q)=\left\{f\in\mathcal{H}(\D)\ :\rho_{p,q}(f)<+\infty\right\}$$
	where
		\begin{equation*}
	\begin{split}
	\rho_ {p,q}(f)&=\left(\frac{1}{2\pi}\int_{0}^{2\pi} \left(\int_{0}^{1} |f(r e^{i t})|^p \ dr \right)^{q/p}dt\right)^{1/q}, \quad \text{ if } p,q<+\infty.
	\end{split}
	\end{equation*}
\end{definition}

	It is easy to see that $RM(p,q)$ endowed with the norm $\rho_{p,q}$ is a Banach space whenever $p,q\geq 1$.

One example of functions in $RM(p,q)$, that we will play an important role in the main result of this work, is the following family of functions.

\begin{proposition}\label{testfunction1} Let $0< p,q\leq +\infty$. Let $\alpha\in \D$ and $\beta>\frac{1}{p}+\frac{1}{q}$ then
	\begin{align*}
	\rho_{p,q}((1-\overline{\alpha}z)^{-\beta})\asymp (1-|\alpha|)^{\frac{1}{p}+\frac{1}{q}-\beta},
	\end{align*}
	where we are using the main branch of the logarithm to define $w^{-\beta}$.
	We underline that the equivalent constants depend on $p$,$q$ and $\beta$, but not on $\alpha$.
\end{proposition}
\begin{proof}
	Let $0<p,q<+\infty$. We can assume without loss of generality that $\alpha\in [0,1)$. Moreover, we can assume that $1/2\leq \alpha<1$.
	
	Let us estimate the quantity $|1-\alpha re^{i\theta}|^2$ for points around $1$. If $0<\theta<1-\alpha$ and $1/2<r<1$, then $|1-\alpha re^{i\theta}|^2\asymp (1-r\alpha)^2$. If $1/2>\theta>1-\alpha$, then
	\begin{align*}
	|1-\alpha re^{i\theta}|^2\asymp 
	\begin{cases}
	\theta^2,  & 1>r\geq \frac{1-\theta}{\alpha},\\
	(1-r\alpha)^2, & 1/2< r\leq \frac{1-\theta}{\alpha}.\\
	\end{cases}
	\end{align*}
	
	First of all, let us see that $\rho_{p,q}((1-\overline{\alpha}z)^{-\beta})\lesssim (1-|\alpha|)^{\frac{1}{p}+\frac{1}{q}-\beta}$ if $\beta>\frac{1}{p}+\frac{1}{q}$. Using the symmetry in $\theta$ and the monotonicity in $\theta$ and $r$, we have
	
	\begin{align*}
	&\int_{0}^{2\pi}\left(\int_{0}^{1} \frac{dr}{|1-\alpha re^{i\theta}|^{\beta p}}\right)^{q/p}\ d\theta  \leq 2^{2+q/p}\pi \int_{0}^{1/2}\left(\int_{1/2}^{1} \frac{dr}{|1-\alpha re^{i\theta}|^{\beta p}}\right)^{q/p}\ d\theta\\
	\end{align*}
	Therefore,
	\begin{align*}
	&\int_{0}^{2\pi}\left(\int_{0}^{1} \frac{dr}{|1-\alpha re^{i\theta}|^{\beta p}}\right)^{q/p}\ d\theta\\
	&\quad\lesssim\left(\int_{0}^{1-\alpha}\left(\int_{1/2}^{1} \frac{dr}{|1-\alpha re^{i\theta}|^{\beta p}}\right)^{q/p}\ d\theta+\int_{1-\alpha}^{1/2}\left(\int_{1/2}^{1} \frac{dr}{|1-\alpha re^{i\theta}|^{\beta p}}\right)^{q/p}\ d\theta\right)\\
	&\quad\leq\int_{1-\alpha}^{1/2}\left(\int_{\frac{1-\theta}{\alpha}}^{1} \frac{dr}{\theta^{\beta p}}+\int_{1/2}^{\frac{1-\theta}{\alpha}} \frac{dr}{(1-r\alpha)^{\beta p}}\right)^{q/p}\ d\theta+\int_{0}^{1-\alpha}\left(\int_{1/2}^{1} \frac{dr}{(1-r\alpha)^{\beta p}}\right)^{q/p}\ d\theta\\
	&\quad\leq  \int_{1-\alpha}^{1/2}\left( \frac{\theta-(1-\alpha)}{\alpha\theta^{\beta p}}+\frac{1}{\alpha(\beta p-1)}\left(\frac{1}{\theta^{\beta p-1}}-\frac{1}{(1-\alpha/2)^{\beta p-1}}\right)\right)^{q/p}\ d\theta\\
	&\qquad+\int_{0}^{1-\alpha}\left(\frac{1}{\alpha(\beta p-1)}\left(\frac{1}{(1-\alpha)^{\beta p-1}}-\frac{1}{(1-\alpha/2)^{\beta p-1}}\right)\right)^{q/p}\ d\theta \displaybreak\\
	&\quad\leq \left(\frac{\beta p}{\alpha(\beta p-1)}\right)^{q/p} \left(\int_{1-\alpha}^{1/2} \frac{1}{\theta^{\beta q-q/p}}\ d\theta +\frac{1}{(1-\alpha)^{\beta q-q/p-1}}\right) \\
	& \quad\leq \left(\frac{\beta p}{\alpha(\beta p-1)}\right)^{q/p} \left(\frac{\beta q-q/p}{\beta q-q/p-1}\right)(1-\alpha)^{1+q/p-\beta q}.
	\end{align*}
	
	Now, we will show that $\rho_{p,q}((1-\overline{\alpha}z)^{-\beta})\gtrsim (1-|\alpha|)^{\frac{1}{p}+\frac{1}{q}-\beta}$ if $\beta>\frac{1}{p}+\frac{1}{q}$.
	Since for $0<\theta<1-\alpha$ and $0\leq r<1$ one have that $|1-\alpha re^{i\theta}|\lesssim (1-\alpha r)$, we have
	
	\begin{align*}
	\int_{0}^{2\pi}\left(\int_{0}^{1} \frac{dr}{|1-\alpha re^{i\theta}|^{\beta p}}\right)^{q/p}\ d\theta & \gtrsim \int_{0}^{1-\alpha}\left(\int_{0}^{1} \frac{dr}{(1-\alpha r)^{\beta p}}\right)^{q/p}\ d\theta\\
	&\geq \frac{1}{(\beta p-1)^{q/p}\alpha^{q/p}}\int_{0}^{1-\alpha}\left(\frac{1}{(1-\alpha)^{\beta p-1}}-1\right)^{q/p} d\theta\\
	&\geq \left(\frac{1-(1/2)^{\beta p-1}}{(\beta p-1)\alpha}\right)^{q/p} (1-\alpha)^{1+q/p-\beta q}.
	\end{align*}
	
	If $p$ or $q$ is $\infty$, the proof follows similarly.
\end{proof}

Now, we continue with the definition of a particular case of the mixed normed spaces. Bearing in mind the $RM(p,q)$ spaces, they are defined interchanging the order of integration, that is, they are formed by those holomorphic functions in $\D$ such that taking the $q$-norm in every circle and then the $p$-norm, the result is a finite number. More precisely, 

\begin{definition}
	Let $0<p,q< +\infty$. We define the mixed norm spaces
	\begin{align*}
	H^{q,p}=\{f\in\mathcal{H}(\D)\ :\ \|f\|_{H^{q,p}}<+\infty\}
	\end{align*}
	where
	\begin{align*}
	&\|f\|_{H^{q,p}}=\left(\int_{0}^{1}\left(\int_{0}^{2\pi} |f(re^{i\theta})|^{q}\ \frac{d\theta}{2\pi}\right)^{p/q}\ dr\right)^{1/p}.
		\end{align*}
\end{definition}

By Fubini's theorem, it is clear that
$$\|f\|_{H^{p,p}}=\rho_{p,p}(f).$$
Thus, it is a natural question to analyse what happens when $p\neq q$. This is the aim of this work. 

Along the same lines as for the $RM(p,q)$ spaces, we have one example of functions that belong to these mixed norm spaces is the following family of analytic functions.
\begin{proposition}{\cite[Proposition 2, p. 947]{Arevalo}}\label{testfunction2} Let $0< p,q<+\infty$. Let $\alpha\in \D$ and $\beta>\frac{1}{p}+\frac{1}{q}$ then
	\begin{align*}
	\|((1-\overline{\alpha}z)^{-\beta})\|_{H^{q,p}}\asymp (1-|\alpha|)^{\frac{1}{p}+\frac{1}{q}-\beta},
	\end{align*}
	where we are using the main branch of the logarithm to define $w^{-\beta}$.
	We underline that the equivalent constants depend on $p$,$q$ and $\beta$, but not on $\alpha$.
\end{proposition}

\section{Containment relationships}
We proceed with the presentation of the containment relationship between the $RM(p,q)$ spaces and a particular case of the mixed norm spaces, which we have defined above.

Bearing in mind the standard argument of \cite[Corollary 4.8, p. 29]{Aguilar-Contreras-Piazza_1}, one can prove an analogous duality result for the mixed norm spaces by means of the boundedness of the weighted Bergman projection:

\begin{proposition}{\cite[Corollary 7.3.4, p. 153]{vukotic_multiplier_2016}}\label{vukotic}
	Let $1\leq p, q< +\infty$ and $1/p<\gamma+1$. Then the operator
	\begin{align*}
	P_{\gamma}f(z)=(\gamma+1) \int_{\D} \frac{(1-|w|^2)^{\gamma}f(w)}{(1-z\overline{w})^{2+\gamma}}\ dA(w)
	\end{align*}	
	is a bounded projection mapping $L^{q,p}$ onto $H^{q,p}$, where $L^{q,p}$ is the corresponding space of equivalence classes of measurable functions. In particular ($\gamma=0$), we have the Bergman projection $P$ maps  $L^{q,p}$ onto $H^{q,p}$, when $1<p<+\infty$.
\end{proposition}


The proof of the next result follows the same scheme as for the spaces $RM(p,q)$   \cite[Corollary 4.8, p. 29]{Aguilar-Contreras-Piazza_1}.

\begin{proposition}\label{dualmixednormspace}
	Let $1<p,q<+\infty$. Then $(H^{q,p})^{\ast} \cong H^{q',p'}$, where $\frac{1}{p}
	+\frac{1}{p'}=1$ and $\frac{1}{q}
	+\frac{1}{q'}=1$. The antiisomorphism between $H^{q',p'}$ and $(H^{q,p})^{\ast}$  is given by the operator 
	$$g\mapsto \lambda_g$$
	where $\lambda_g$ is defined by 
	\begin{align*}
	\lambda_{g}(f)=\int_{\D} f(z)\overline{g(z)}\ dA(z), \quad f\in H^{q,p}. 
	\end{align*}
\end{proposition}
Now we are ready to prove Theorem \ref{main}.


\begin{proof} [Proof of Theorem \ref{main}]
	\textit{a)} If $p\geq q$ then $RM(p,q)\subset H^{q,p}$, because using Minkowski's integral inequality we have 
	\begin{align*}
	\left(\int_{0}^{1} \left(\int_{0}^{2\pi} |f(re^{i\theta})|^{q}\ \frac{d\theta}{2\pi}\right)^{p/q}\ dr\right)^{1/p}\leq \left(\int_{0}^{2\pi}\left(\int_{0}^{1}|f(re^{i\theta})|^{p}\ dr\right)^{q/p}\ \frac{d\theta}{2\pi}\right)^{1/q}.
	\end{align*}
	
	Let us see that $RM(p,q)\varsubsetneq H^{q,p}$ if $p>q$.
	
	Consider the functions $$u_{\delta}(z)=\frac{\delta}{(1+\delta-z)^{1+\frac{1}{p}+\frac{1}{q}}},\quad z\in\D,$$
	where $0<\delta<1/2$. One can see that $\|u_\delta\|_{H^{q,p}}\asymp \rho_{p,q}(u_\delta)\asymp 1$, because, for $\alpha\in \D$,
	$$\|(1-\overline{\alpha}z)^{-1-\frac{1}{p}-\frac{1}{q}}\|_{H^{q,p}}\asymp (1-|\alpha|)^{-1}$$
	(see Proposition~\ref{testfunction2}) and  $$\rho_{p,q}((1-\overline{\alpha}z)^{-1-\frac{1}{p}-\frac{1}{q}})\asymp (1-|\alpha|)^{-1}$$ 
	(see Proposition~\ref{testfunction1}).
	
	Let $\{\delta_n\}$ be a sequence of positive numbers such that $n^2\delta_n<\frac{1}{4}$ for all $n\geq 1$, $\delta_{n}/n^{2p}>(n+1)\delta_{n+1}^{1/2}$ for all $n\geq 1$, and $\sum_{j=1}^{\infty} j^2\delta_j<2$. Define the sets
	$$A_{n}:=\left\{ re^{i\theta} :|\theta-\theta_{n}|\leq n^{2}\delta_n\ ,\ r\in[1-n\delta_n^{1/2},1-\delta_n/n^{2p}]\right\},$$
	where $\theta_{1}=\delta_1$  and $\theta_{n}-\theta_{n-1}=\frac{1}{n^2}+(n-1)^2\delta_{n-1}+n^2\delta_n$, for $n\geq 2$. Observe that $A_{n}=\{re^{i\theta}\ :\ r\in I_{n},\ \theta\in J_{n}\}$ where the sets $\{I_n\}$ are pairwise disjoint and so are the sets $\{J_n\}$.
	
	Let us check that $\rho_{p,q}(u_{\delta_n}(ze^{-i\theta_{n}})\chi_{\D\setminus A_n}(z))\lesssim \frac{1}{n^2}$. Firstly, notice that
	\begin{align}\label{eqp1mn}
	&\rho_{p,q}\left(u_{\delta_n}(ze^{-i\theta_{n}})\chi_{\{0<|w|<1-n\delta_n^{1/2}\}}(z)\right)\\
	&\quad=\left(\int_{0}^{2\pi} \left(\int_{0}^{1-n\delta_n^{1/2}}\frac{\delta_n^p}{|1+\delta_n-re^{i\theta}|^{p\left(1+\frac{1}{p}+\frac{1}{q}\right)}}   \ dr \right)^{q/p}\frac{d\theta}{2\pi}\right)^{1/q}\nonumber\\
	&\quad \leq  \left(\int_{0}^{1-n\delta_n^{1/2}}\frac{\delta_n^p}{(1+\delta_n-r)^{p\left(1+\frac{1}{p}+\frac{1}{q}\right)}}   \ dr \right)^{1/p}\leq \frac{\delta_n}{(n\delta_n^{1/2}+\delta_n)^{1+\frac{1}{q}}}\nonumber\\
	&\quad\leq \frac{\delta_n}{(n\delta_n^{1/2})^{1+\frac{1}{q}}}\leq \frac{1}{n^{2}}.\nonumber
	\end{align} 
	
	In the next inequalities we use that, for $|\theta|\leq 1$ and $r\geq 1-n\delta_{n}^{1/2}$
	\begin{align*}
	|1+\delta_n-re^{i\theta}|^2&=(1+\delta_{n}-r)^2+2r(1+\delta_{n})(1-\cos(\theta))\geq(1+\delta_{n}-r)^{2}+\frac{r\theta^2}{2}\\
	&\geq (1+\delta_{n}-r)^{2}+(1-n\delta_{n}^{1/2})\frac{\theta^2}{2}\geq (1+\delta_{n}-r)^{2}+\frac{\theta^2}{4}.
	\end{align*}
	Thus, it follows 
	\begin{align*}
	&\rho_{p,q}^{q}\left(u_{\delta_n}(ze^{-i\theta_{n}})\chi_{\{ re^{i\theta}\ :\ 1-n\delta_n^{1/2}<r<1,\ |\theta-\theta_{n}|> n^2\delta_n \}}(z)\right)\\
	&\quad= 2\int_{n^2\delta_n}^{\pi}\left(\int_{1-n \delta_n^{1/2}}^{1}\frac{\delta_n^p}{|1+\delta_n-re^{i\theta}|^{p\left(1+\frac{1}{p}+\frac{1}{q}\right)}}   \ dr \right)^{q/p} \frac{d\theta}{2\pi}\nonumber\\
	&\quad\leq 2\left(\frac{\pi-n^2\delta_{n}}{1-n^2\delta_{n}}\right)\int_{n^2\delta_n}^{1}\left(\int_{1-n \delta_n^{1/2}}^{1}\frac{\delta_n^p}{|1+\delta_n-re^{i\theta}|^{p\left(1+\frac{1}{p}+\frac{1}{q}\right)}}   \ dr \right)^{q/p} \frac{d\theta}{2\pi},\nonumber
	\end{align*}
	since the inner integral is a decreasing function in $\theta$. Therefore,
	\begin{align}\label{eqp2mn}
	&\rho_{p,q}^{q}\left(u_{\delta_n}(ze^{-i\theta_{n}})\chi_{\{ re^{i\theta}\ :\ 1-n\delta_n^{1/2}<r<1,\ |\theta-\theta_{n}|> n^2\delta_n \}}(z)\right) \\
	&\quad\leq 2^{3q+1}  \int_{n^2\delta_n}^{n^2\delta_n + n\delta_n^{1/2}}\left(\int_{1-n \delta_n^{1/2}}^{1+\delta_n-\theta}\frac{\delta_n^p}{(1+\delta_n-r)^{p\left(1+\frac{1}{p}+\frac{1}{q}\right)}}   \ dr\right.\nonumber\\
	&\qquad+\left.\int_{1+\delta_n-\theta}^{1}\frac{\delta_n^p}{\theta^{p\left(1+\frac{1}{p}+\frac{1}{q}\right)}}   \ dr \right)^{q/p} d\theta\nonumber\\
	&\qquad+2^{3q+1} \int_{n^2\delta_n + n\delta_{n}^{1/2}}^{1} \left(\int_{1-n \delta_n^{1/2}}^{1}\frac{\delta_n^p}{\theta^{p\left(1+\frac{1}{p}+\frac{1}{q}\right)}}   \ dr \right)^{q/p} \ d\theta\nonumber \\
	&\quad\leq 2^{3q+1} \int_{n^2\delta_n}^{n^2\delta_n + n\delta_n^{1/2}}\left(\frac{1}{p\left(1+\frac{1}{q}\right)} \frac{\delta_n^p}{\theta^{p\left(1+\frac{1}{q}\right)}}   +\frac{\delta_n^p}{\theta^{p\left(1+\frac{1}{q}\right)}} \right)^{q/p} d\theta\nonumber\\
	&\qquad+2^{3q+1} \int_{n^2\delta_n + n\delta_{n}^{1/2}}^{1} \frac{\delta_n^{q+\frac{q}{2p}} n^{q/p}}{\theta^{q\left(1+\frac{1}{p}+\frac{1}{q}\right)}} \ d\theta\nonumber \\
	&\quad\leq 2^{3q+1}  2^{q/p}\delta_n^{q} \int_{n^2\delta_n}^{n^2\delta_n + n\delta_n^{1/2}} \frac{1}{\theta^{q\left(1+\frac{1}{q}\right)}}+2^{3q+1} \delta_n^{q} \int_{n^2\delta_n + n\delta_n^{1/2}}^{1} \frac{1}{\theta^{q\left(1+\frac{1}{q}\right)}} \ d\theta\nonumber\\
	&\quad\leq 2^{3q+1+\frac{q}{p}}    \delta_n^{q}\int_{n^2\delta_n}^{1} \frac{1}{\theta^{q\left(1+\frac{1}{q}\right)}} \ d\theta\leq  2^{3q+1+\frac{q}{p}}   \frac{\delta_n^{q}}{(n^2\delta_n)^q}\leq 2^{3q+2} \frac{1}{n^{2q}}.\nonumber
	\end{align}

	Similarly, we obtain	
	\begin{align}\label{eqp3mn}
	&\rho_{p,q}^{q}\left(u_{\delta_n}(ze^{-i\theta_{n}})\chi_{\{ re^{i\theta}\ :\  1-\frac{\delta_n}{n^{2p}}<r<1,\ |\theta-\theta_{n}|<n^2\delta_n\}}(z)\right)\\
	&\quad=2\int_{0}^{n^2\delta_n} \left(\int_{1-\frac{\delta_n}{n^{2p}}}^{1}\frac{\delta_n^p}{|1+\delta_n-re^{i\theta}|^{p\left(1+\frac{1}{p}+\frac{1}{q}\right)}}   \ dr \right)^{q/p}\ \frac{d\theta}{2\pi}\nonumber\\
	&\quad\leq 2\int_{0}^{\delta_n} \left(\int_{1-\frac{\delta_n}{n^{2p}}}^{1}\frac{\delta_n^p}{(1+\delta_n-r)^{p\left(1+\frac{1}{p}+\frac{1}{q}\right)}}   \ dr \right)^{q/p}\ \frac{d\theta}{2\pi}\nonumber\\
	&\qquad+2^{3q+1}\int_{\delta_n}^{n^2\delta_n} \left(\int_{1-\frac{\delta_n}{n^{2p}}}^{1}\frac{\delta_n^p}{\theta^{p\left(1+\frac{1}{p}+\frac{1}{q}\right)}}   \ dr \right)^{q/p}\ \frac{d\theta}{2\pi}\nonumber\\
	&\quad\leq  2\delta_n^q  \left(\frac{\delta_n}{n^{2p}} \right)^{q/p}\int_{0}^{\delta_n} \frac{1}{\delta_n^{q\left(1+\frac{1}{p}+\frac{1}{q}\right)}} \frac{d\theta}{2\pi}+2^{3q+1}\delta_n^q  \left(\frac{\delta_n}{n^{2p}} \right)^{q/p}\int_{\delta_n}^{n^2\delta_n} \frac{1}{\theta^{q\left(1+\frac{1}{p}+\frac{1}{q}\right)}} \ \frac{d\theta}{2\pi}\nonumber.   \\
	&\quad\leq \frac{8^{q+1}}{2\pi}\delta_n^q  \left(\frac{\delta_n}{n^{2p}} \right)^{q/p} \frac{1}{\delta_n^{q\left(1+\frac{1}{p}\right)}}=\frac{8^{q+1}}{2\pi}\frac{1}{n^{2q}}.\nonumber
	\end{align}
	Using inequalities \eqref{eqp1mn}, \eqref{eqp2mn}, and \eqref{eqp3mn} we deduce $$\|u_{\delta_n}(ze^{-i\theta_{n}})\chi_{\D\setminus A_n}(z)\|_{H^{q,p}}\leq \rho_{p,q}(u_{\delta_n}(ze^{-i\theta_{n}})\chi_{\D\setminus A_n}(z))\lesssim \frac{1}{n^2}.$$
	
	Now, by the very definition of $A_n$, the sets $\{A_{n}\}$ are pairwise disjoint. Define the functions $g_n(z)=u_{\delta_{n}}(ze^{-i\theta_{n}}) \chi_{\D\setminus A_{n}}(z)$ and $f_n(z)=u_{\delta_{n}}(ze^{-i\theta_{n}}) \chi_{A_{n}}(z)$ such that $u_{\delta_{n}}(ze^{-i\theta_{n}})=f_{n}(z)+g_{n}(z)$. As we have seen, we have that $\rho_{p,q}(g_n)\lesssim \frac{1}{n^2}$ and $\|g_n\|_{H^{q,p}}\lesssim \frac{1}{n^2}$  for $n\in\N$. In addition, one can see that $\rho_{p,q}(f_n)\asymp\|f_n\|_{H^{q,p}}\asymp 1$ because $\rho_{p,q}(u_{\delta_{n}}(ze^{-i\theta_{n}}))\asymp \|u_{\delta_{n}}(ze^{-i\theta_{n}})\|_{H^{q,p}}\asymp 1$.

	Given a measure space $(\Omega, \mathcal{A},\mu)$. We have that for any sequence of measurable functions $h_{n}: \Omega \rightarrow \C$ whose supports are pairwise disjoint, it holds that
	$$
	\int_{\Omega} \bigg|\sum_{n} h_{n}(w)\bigg|^{s}\ d\mu(w)=\sum_{n}\int_{\Omega} |h_{n}(w)|^{s}\ d\mu(w)
	$$
	for all $s>0$. Then, using this fact twice (one in each variable, first with $s=q$ and then with $s=p/q$), we obtain
	\begin{align*}
	&\left\|\sum \alpha_n f_n\right\|_{H^{q,p}}=\left(\int_{0}^{1}\sum |\alpha_n|^{p} \left(\int_{0}^{2\pi}\left| f_n(re^{i\theta})\right|^{q}\frac{d\theta}{2\pi}\right)^{p/q}\ dr\right)^{1/p}\\
	&\quad=\left(\sum |\alpha_n|^{p} \|f_n\|_{H^{q,p}}^{p}\right)^{1/p}\asymp \left(\sum |\alpha_n|^{p}\right)^{1/p}.
	\end{align*}
	By the same reason,
	\begin{align*}
	&\rho_{p,q}\left(\sum \alpha_n f_n\right)=\left(\int_{0}^{2\pi} \sum |\alpha_n|^q\left(\int_{0}^{1}\left| f_n(re^{i\theta})\right|^{p}\ dr\right)^{q/p}\frac{d\theta}{2\pi}\right)^{1/q}\\
	&\quad=\left(\sum |\alpha_n|^{q} \rho_{p,q}^{q}(f_n)\right)^{1/q}\asymp \left(\sum |\alpha_n|^{q}\right)^{1/q}.
	\end{align*}
	Hence, if we consider the function $F_{m}(z):=\sum_{n=1}^{m} u_{\delta_{n}}(ze^{-i\theta_{n}})$ we obtain that
	\begin{align*}
	\rho_{p,q}(F_m)&\leq \rho_{p,q}\left(\sum_{n=1}^{m} f_n\right)+\rho_{p,q}\left(\sum_{n=1}^{m} g_n\right)\lesssim m^{1/q}+\sum_{n=1}^{m}\frac{1}{n^2}\\
	&\leq m^{1/q} + \frac{\pi^2}{6}\leq \left(1+\frac{\pi^2}{6}\right)m^{1/q}
	\end{align*}
	and 
	\begin{align*}
	\rho_{p,q}(F_m)\geq \rho_{p,q}\left(\sum_{n=1}^{m} f_n \right)-\rho_{p,q}\left(\sum_{n=1}^{m} g_n\right)\gtrsim m^{1/q}
	\end{align*}
	for $m$ big enough. So that $\rho_{p,q}(F_m)\asymp m^{1/q}$. In the same way for the norm in $H^{q,p}$, it follows that $\|F_m\|_{H^{q,p}}\asymp m^{1/p}$ using the same argument.
	
	Therefore, if it were true that $RM(p,q)=H^{q,p}$, then we would have $\rho_{p,q}(F_m)\asymp \|F_m\|_{H^{q,p}}$. But if $p>q$ this is impossible, because this implies that $m^{1/p}\asymp m^{1/q}$ for all $m\in\N$. Thus, we conclude \textit{a)}. \newline
	
	\textit{b)} If $q\geq p$ then $ H^{q,p}\subset RM(p,q)$, using Minkowski's integral inequality  we have 
	\begin{align*}
	\left(\int_{0}^{2\pi}\left(\int_{0}^{1}|f(re^{i\theta})|^{p}\ dr\right)^{q/p}\ \frac{d\theta}{2\pi}\right)^{1/q}\leq \left(\int_{0}^{1} \left(\int_{0}^{2\pi} |f(re^{i\theta})|^{q}\ \frac{d\theta}{2\pi}\right)^{p/q}\ dr\right)^{1/p}.
	\end{align*}
	Let us see that $H^{q,p}\neq RM(p,q)$ for $q> p$. Assume that $H^{q,p}= RM(p,q)$ then $(H^{q,p})^{\ast}= (RM(p,q))^{\ast}$. By \cite[Corollary 4.8, p. 29]{Aguilar-Contreras-Piazza_1} and Proposition~\ref{dualmixednormspace}, we have that $H^{q',p'}= RM(p',q')$ for $p'>q'$. But this contradicts \textit{a)}. So \textit{b)} holds and we are done.
\end{proof}



\end{document}